\documentclass[reqno,12pt]{amsart}
\usepackage{color}
\usepackage{mathtools}
\usepackage{amsfonts}
\usepackage{amssymb}
\usepackage{amsthm}
\usepackage{newlfont}
\usepackage{graphicx}
\usepackage{float}
\usepackage{amsmath}
\usepackage{geometry}
\usepackage{tikz}
\usetikzlibrary{matrix,shapes,arrows,positioning,chains}

\numberwithin{equation}{section}
\newtheorem{thm}{Theorem}[section]
\newtheorem{lem}{Lemma}[section]
\newtheorem{rem}{Remark}[section]

\newtheorem{defn}{Definition}[section]

\newtheorem{example}{Example}[section]

\newcommand{\eps}{\varepsilon}

\newcommand{\ed}{\end {document}}

\newcommand{\ka}{\kappa}

\begin{document}

\title{On symmetry breaking of Allen--Cahn}

\author[D. Li]{ Dong Li}
\address{D. Li, Department of Mathematics, the Hong Kong University of Science \& Technology, Clear Water Bay, Kowloon, Hong Kong}
\email{mpdongli@gmail.com}

\author[C.Y. Quan]{Chaoyu Quan}	
\address{C.Y. Quan, SUSTech International Center for Mathematics, Southern University of Science and Technology,
	Shenzhen, P.R. China}
\email{quancy@sustech.edu.cn}

\author[T. Tang]{Tao Tang}
\address{T. Tang, Guangdong Provincial Key Laboratory of Computational Science and Material Design, Southern University of Science and Technology, Shenzhen, P.R. China; and Division of Science and Technology, BNU-HKBU United International College, Zhuhai, P.R. China}
\email{ttang@uic.edu.cn}

\author[W. Yang]{Wen Yang}
\address{\noindent W. Yang,~Wuhan Institute of Physics and Mathematics, Chinese Academy of Sciences, P.O. Box 71010, Wuhan 430071, P. R. China; Innovation Academy for Precision Measurement Science and Technology, Chinese Academy of Sciences, Wuhan 430071, P. R. China.}
\email{wyang@wipm.ac.cn}

\begin{abstract}
This paper is concerned with numerical solutions for the Allen-Cahn equation with standard double well potential and periodic boundary conditions. Surprisingly it is found that using standard numerical discretizations with high precision computational
solutions may converge  to completely incorrect steady
states.  This happens for very smooth initial data and state-of-the-art algorithms.
We analyze this phenomenon and showcase the resolution of this problem by a new
symmetry-preserving filter technique.
We develop a new theoretical framework and rigorously prove the convergence to steady states for the filtered solutions.
\end{abstract}

\maketitle
\section{Introduction}
{\bf A curious experiment}.
Consider the following 1D Allen--Cahn on the periodic torus $\mathbb T=[-\pi, \pi]$:
\begin{align}
\label{1.ac}
\begin{cases}
\partial_t u = \ka^2 \partial_{xx} u  - f(u),  \\
u\Bigr|_{t=0} = u_0,
\end{cases}
\end{align}
where $f(u) = u^3 - u$ corresponds to the usual double-well potential, $\kappa^2>0$ is the diffusion coefficient.
The scalar function $u: \; \mathbb T\to \mathbb R$ represents the
concentration of a phase in an alloy and typically has values in the physical range $[-1, 1]$. {There
is by now an extensive literature on the theoretical analysis and numerical simulation of the
Allen-Cahn equation and related phase field models (cf. \cite{aac2001,ac2000,bbg2000,bfrw1997,bhm2000,dkw2011,wang2017}).} 
For the system \eqref{1.ac} we take the initial data $u_0$ as an odd function of $x$
(here we tacitly ``lift" the periodic function $u$ to be defined on the whole
real axis so that oddness can be defined in the usual way). For example, one can take
$u_0(x) = \sin(x)$.
Denote
\begin{align}
u_{\infty}(x) = \lim_{t\to \infty} u(x,t).
\end{align}
Since $u_0$ is odd and smooth, it is clear that the odd symmetry should be preserved for all time.
In particular, the final state $u_{\infty}$ must be a periodic odd function of $x$.
However  to our surprise, standard  numerical experiments show that this parity property may be
lost in not very long time simulations.
For example, by using the finite difference method with $N_{p} = 128$ nodes for space discretization and the first-order implicit-explicit method with time step $\tau = 0.01$ for time discretization, we
 find that $u$ tends to $\pm 1$ when $t$ is suitably large (see the left-hand side of Figure \ref{fig:intro1}).
Obviously, $\pm 1$ are not the correct steady states as the oddness is not preserved.

To check the fidelity of the numerical scheme and rule out the issues connected with inaccurate numerical discretization, we first test larger $N = 4096$ and smaller $\tau  = 10^{-5}$.
It turns out that for this case, the oddness is preserved up to around $t\approx 32$,
after which the solution  $u$ lost its oddness apparently and tends to $1$ quickly. See the right-hand side of Figure \ref{fig:intro1}.

\vskip .3cm
{\bf How about other numerical schemes?}
We test different first and second order (in time) numerical schemes to find that all these methods suffer the same non-physical phenomenon, see Figure \ref{fig:intro2}.  This appears to be quite
a serious issue, since it shows that all these standard numerical implementations may not
be credible in not very long time simulations.

\begin{figure}[!h]
\centering
(a) \includegraphics[width=0.43\textwidth]{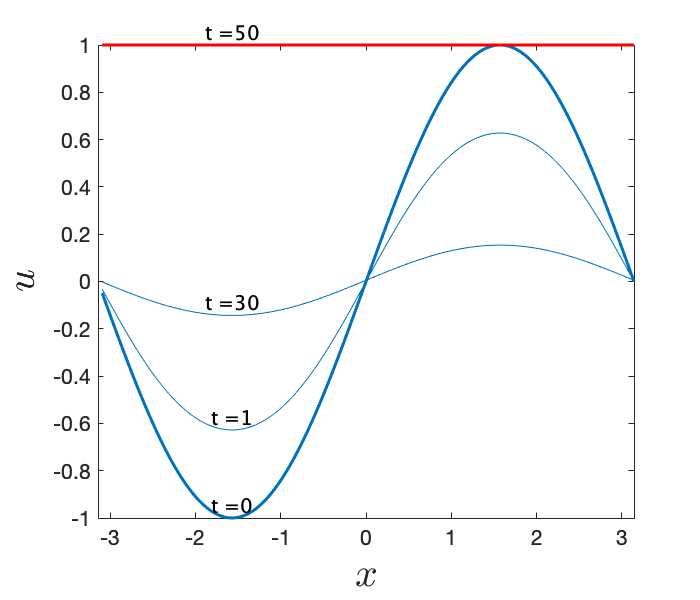}
(b) \includegraphics[width=0.43\textwidth]{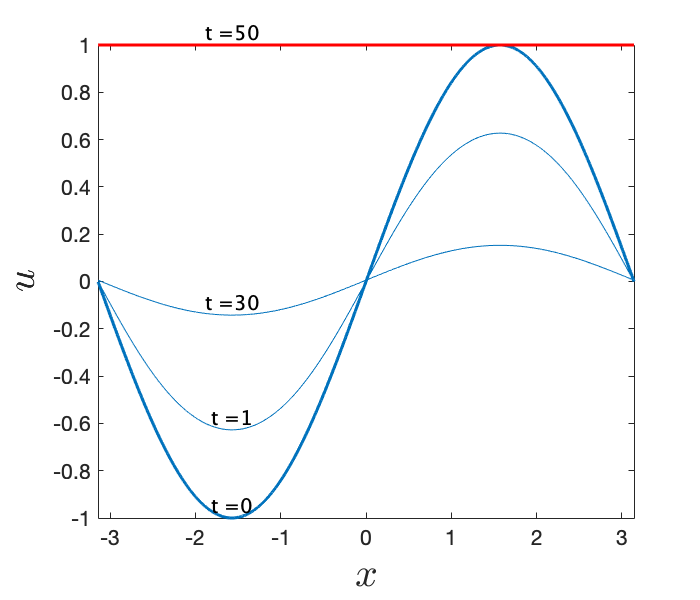}
\caption{\small Dynamics of the Allen-Cahn equation using the first-order implicit-explicit scheme with different discretization parameters. (a): $\ka=1$, $N_p = 128$ and $\tau= 0.01$, and (b): $\ka=1$, $N_p = 4096$ and $\tau=10^{-5}$. }\label{fig:intro1}
\end{figure}

\begin{figure}[!h]
\centering
\includegraphics[trim=1in 0.4in 1in 0.in,clip,width=0.99\textwidth]{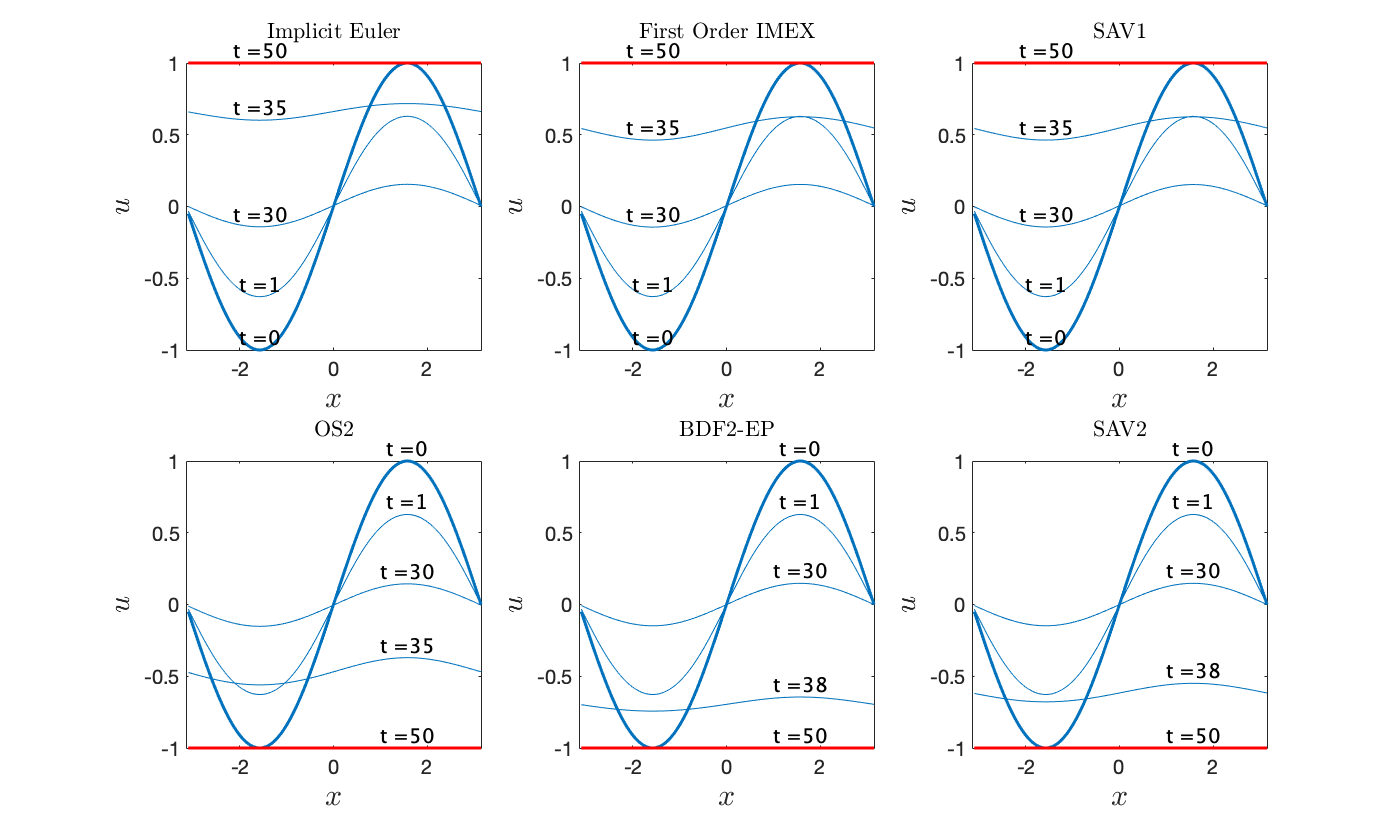}
\caption{\small Dynamics of the Allen-Cahn equation using different numerical schemes ($\ka=1$, $N_p = 128$, $\tau= 0.01$): implicit Euler method \cite{xuj2019}, the first-order IMEX method, the second-order operator splitting methods \cite{cheng2015}, BDF2 extrapolation method \cite{xu2006}, the first- and second-order scalar auxiliary variable methods \cite{shen2018}. }\label{fig:intro2}
\end{figure}

{\bf Outline of the paper.}
The purpose of this work is to rectify this issue for some special initial conditions by introducing a parity-preserving filter.
The importance of this filter is that it eliminates the aforementioned
unwanted projections into the unstable directions at each iteration.
In the simplest situation of odd initial data such as $\sin x$, we dynamically force all the  Fourier coefficient corresponding to cosine Fourier modes of the numerical solution to remain strictly zero after each iteration.
For initial data such as $\sin(Lx)$ with $L\ge 2$ an integer, one can
check that the spectral gap $L$ is preserved in time, and our proposed filter forces the Fourier
coefficients to retain the spectral gap as well as the parity conditions.
Furthermore, we rigorously prove that the filtered solution will converge to the true steady state.

The rest of this paper is organized as follows.
In Section 2 we classify the steady states and analyze their profiles.
In Sections 3 an 4 we analyze the symmetry breaking phenomenon and show how the filtering approach can resolve this issue.
In Section 5, by providing more numerical examples we discuss in more
details about various scenarios where the parity can be preserved or lost when the filter is in effect or turned off.
It will demonstrated that although the algorithm and analysis are provided for 1D Allen-Cahn the effectiveness of the filtering technique is observed for the 2D Allen-Cahn equation and 2D Cahn-Hilliard equation.
In the last section we give some concluding remarks.

\vspace{0.5cm}
\section{Properties of steady states}\label{sect2}
We consider the steady state solution of \eqref{1.ac} which satisfies
\begin{equation}
\label{2.ac}	
\kappa^2u''+u-u^3=0, \quad x\in\mathbb{T}.
\end{equation}
%
%
We first discuss a special class of steady states, called {\bf odd zero-up ground states}.

\begin{defn}[Odd zero-up ground state]
\label{def3.1}
For any $0<\ka<1$, $U_\ka$ is an odd zero-up ground state of \eqref{2.ac} if $U_\ka$ is odd, $U'_\ka(x) >0$ on $\left[0,\frac \pi 2\right)$, and $U'_\ka(x) <0$ on $\left(\frac \pi 2,\pi\right]$.
\end{defn}

To fix the notation, let us recall the  energy functional associated with \eqref{1.ac}:
\begin{align} \label{E_ground01}
E(u) = \int_{\mathbb T} \left(
\frac {\ka^2} 2 (\partial_x u )^2 + \frac 14 (1-u^2)^2 \right)dx.
\end{align}

{
\begin{rem}
By a result from \cite{LQTW_2}, for each $0<\ka<1$ and excluding
the trivial states $u\equiv\pm 1$, $U_{\ka}$ has the lowest energy among
the steady state solutions, namely,
\begin{align}
  E_{\ka}^{(0)}&=\inf_{u \in \mathcal F} E(u)  =E(U_{\ka}),
\end{align}
where the class $\mathcal F$ of admissible functions is given by
\begin{align} \label{eScons}
\mathcal F=\{u~|
~u\in H^1(\mathbb T)~\mathrm{solves}~\eqref{2.ac}~\mathrm{and}~|u|<1~\mathrm{for}~x\in\mathbb T\}.
\end{align}
Furthermore, if $v\in \mathcal F$ attains $E_{\kappa}$, then for some $x_0 \in
\mathbb T$, we have $v(x)= \pm U_{\ka}(x-x_0)$.
Note that the constraint $|u|<1$ in
\eqref{eScons} is to exclude the trivial minimizers $u=\pm 1$. Another way is to impose
symmetry. For example $U_{\ka}$ or $-U_{\ka}$ are the only minimizers amongst all odd $2\pi$-periodic
functions.

\end{rem}
}

It is possible to classify all steady states to \eqref{2.ac} by using $U_\ka$. 
For any $0<\ka<1$, define $m_{\ka}\ge 1$ as the unique integer such that
\begin{align}\label{eq:2.4}
	\frac 1 {m_{\ka}+1} \le \ka  <\frac 1 {m_{\ka}}.
\end{align}

\begin{thm}[\cite{LQTW_2}]
\label{2.mainthm}
Let $0<\ka<1$. If $u$ is a bounded solution to \eqref{2.ac}, then only one of the following holds.
	\begin{enumerate}
		\item [(1)] $u\equiv \pm 1$ or $u\equiv 0$.
		\item [(2)] $u = \pm U_{j\ka} (jx +c)$, for some integer $j\in \{1,2,\ldots,m_\ka\}$ and some constant $c$. Here note that $0<j\ka<1$ by \eqref{eq:2.4} and $U_{j\ka}(j\cdot)$ can be viewed as a rescaled function of $U_{j\ka}(\cdot)$, where $u(\cdot) = U_{j\ka}(\cdot)$ corresponds to the odd zero-up ground state of the equation \begin{equation}
	j^2\kappa^2u''+u-u^3=0, \quad x\in\mathbb{T}.
\end{equation}
	\end{enumerate}
\end{thm}
\begin{rem}
For $\ka\geq 1$, it can be shown (see \cite{LQTW_2}) that
if $u$ is a bounded steady state solution to \eqref{2.ac}, then $u\equiv 0$ or $\pm 1$.
\end{rem}

Roughly speaking, the significance of the odd zero-up ground state $U_{\ka}$ introduced in
Definition \ref{def3.1} is that, the set $\{U_{\ka} \}_{0<\ka<1} $ gives a clean description 
of all possible bounded steady states of \eqref{2.ac}. Indeed by Theorem \ref{2.mainthm},
any other bounded steady state must either be $0$, $\pm 1$ or a rescaled, translated copy
from the catalogue $\{U_{\ka} \}_{0<\ka<1}$.  Therefore it is of fundamental importance to
study the odd zero-up ground states $\{U_{\ka} \}_{0<\ka<1}$.

We now characterize the detailed profile of the odd zero-up ground state of $\eqref{2.ac}$. 
Multiplying $\eqref{2.ac}$ (now $u$ is replaced by $U_\ka$) by $U'_\ka$, integrating the resulting equation from $\frac \pi 2$ to $x$, and using $U_\ka^{\prime}(\frac {\pi} 2) =0$ (a fact from the definition \ref{def3.1}), we have
\begin{equation}\label{2e2}
(U_\ka^{\prime})^2 = \frac 1 { 2 {\ka^2} }
\Big( (U_\ka^2-1)^2 - (N_\ka^2-1)^2 \Big),
\end{equation}
where $N_{\ka} \coloneqq U_\ka(\frac \pi 2)$ is the maximum of $U_\ka$.

By \eqref{E_ground01} and \eqref{2e2},  the ground state energy $E_{\ka}^{(0)}
=E(U_{\ka})$ can be rewritten as
\begin{equation}
\label{E_grd001}
E_{\ka}^{(0)}  = \int_{\mathbb T} \Bigl( \frac 12 (U_{\ka}^2(x)-1)^2 -\frac 14
(N_{\ka}^2-1)^2 \Bigr) dx.
\end{equation}
Since $U_\ka$ is monotonically increasing on $(0,\frac\pi 2)$, we have
\begin{align}\label{eq:ode_Uka}
U_\ka^{\prime}(x) = \frac 1 {\sqrt 2 {\ka} }
\sqrt{ (U_\ka^2-1)^2 - (N_\ka^2-1)^2}, \quad x\in (0,\pi/ 2),
\end{align}
with $U_\ka(0)=0$.
This yields
\begin{align} \label{eq:N_ka}
g(N_\ka) \coloneqq\int_0^{N_\ka} \frac 1 { \sqrt{ (u^2-1)^2 - (1- N_\ka^2)^2} } du = \int_0^{\frac\pi2}\frac{1}{\sqrt{2-N_{\ka}^2(1+\sin^2\theta)}}d\theta=\frac{\pi}{2\sqrt{2}\ka}.
\end{align}
Note that $g(N_k)$ is monotonically increasing on $[0,1)$,
\[
g (0) = \frac {\pi} {2\sqrt 2}, \quad g(N_\ka) \to\infty\;\; \mbox{as} \;\; N_\ka\to 1.
\]
From \eqref{eq:N_ka}, we  derive the constraint
\begin{equation}
0<\ka<1.
\end{equation}

We  have the following monotonicity properties of  $N_\ka$ and $U_\ka$ w.r.t. $\ka$.


\begin{thm}[Monotonicity of $U_\ka$ and $E_\ka^{(0)}$]
If $0<\ka_1<\ka_2< 1$, then
\begin{itemize}
\item [(1)] the monotonicity of $U_\ka$ w.r.t. $\ka$ holds point-wisely
\begin{equation}
U_{\kappa_1}(x)>U_{\kappa_2}(x),\quad x\in\left(0,\pi/2\right];
\end{equation}
\item [(2)] 	the monotonicity of $E_\ka^{(0)}$ w.r.t. $\ka$ also holds
\begin{equation}
E_{\ka_1}^{(0)} < E_{\ka_2}^{(0)}.
\end{equation}
\end{itemize}
Moreover,
\begin{align} \label{E_eplimit00}
\lim_{\ka \to 0} \frac {E_{\ka}^{(0)} } { {\ka} } = \frac 4 3 \sqrt 2.
\end{align}
\end{thm}
\begin{proof}
The proof of  monotonicity can be found in \cite{LQTW_2}.
Here we only show the proof of \eqref{E_eplimit00}.
For $0<\ka\ll 1$, it can be proved (see \cite{LQTW_2}) that $U_\ka$ satisfies
\begin{align} \label{sdy1.2}
0\le  \tanh\left( \frac x {\sqrt 2 \ka } \right)- U_{\ka}(x)\le  \exp\left(-\frac {c} {\ka}\right),
\qquad\forall\, 0\le x\le \frac {\pi}2,
\end{align}
where $c>0$ is an absolute constant.
Then \eqref{E_eplimit00} follows from
a simple computation
\begin{align}
\lim_{\ka \to 0} \frac {E_{\ka}^{(0)} } { {\ka} }  ={\sqrt 2}  \int_{\mathbb R} (\tanh^2 y -1)^2 dy
 = {\sqrt 2} \int_{\mathbb R} (1-\tanh^2 y) \, d (\tanh y)
=\frac 43 {\sqrt 2}.
\end{align}
This completes the proof of \eqref{E_eplimit00}.
\end{proof}
The monotonicity of ground state energy w.r.t. $\ka$, the limit property \eqref{E_eplimit00}, and the monotonicity of $U_\ka$ w.r.t. $\ka$, are verified numerically in Figure \ref{fig:Eeps}.

\begin{figure}[!h]
\centering
\includegraphics[width=0.48\textwidth]{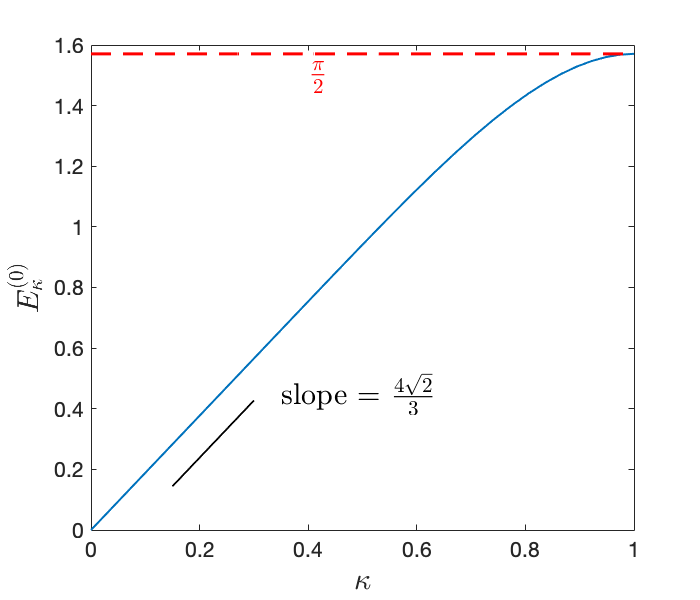}
\includegraphics[width=0.48\textwidth]{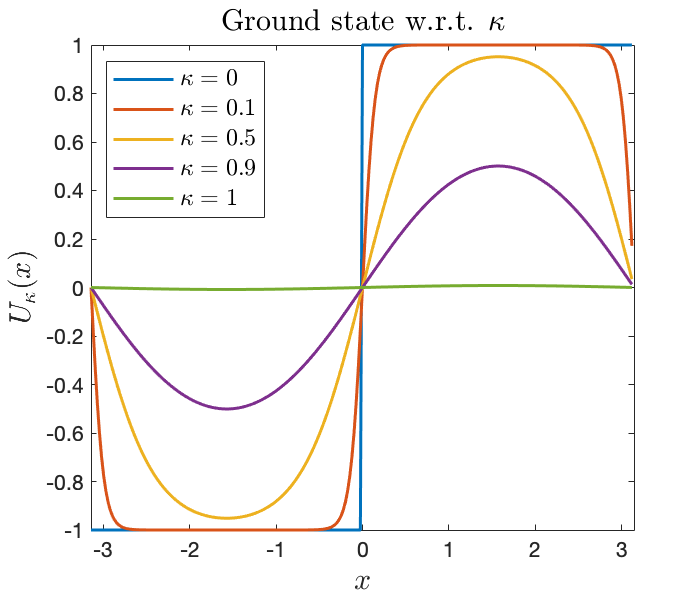}
\caption{$E^{(0)}_\ka$ w.r.t. $\ka\in[0,1]$ and the ground state $U_\ka$ for different $\ka$. }\label{fig:Eeps}
\end{figure}

Note that instead of solving the Allen-Cahn dynamics it is possible to compute the profile of the aforementioned ground state $U_\ka$ in two straightforward ways, namely, determine $N_\ka$ via \eqref{eq:N_ka} by using Newton's iteration method, or solve the ODE \eqref{eq:ode_Uka} defined on $[0,\frac \pi 2]$ by using some standard ODE solvers such as the second-order predictor-corrector method.

\section{Symmetry breaking}
The following sharp convergence result is proved in \cite{LQTW_2}.
\begin{thm}[Convergence for single bump initial data]\label{3.1.thm}
Let $0<\ka<1$.
Assume the initial data $u_0: \, [-\pi, \pi] \to \mathbb R$ is odd, $2\pi$-periodic, and $E(u_0)\le\frac \pi 2$.
Suppose $u_0$ is monotonically increasing on $[0, \frac {\pi}2]$ and $u_0(\pi-x)=u_0(x)$ for
all $\frac {\pi}2 \le x \le \pi$.  Then $u(x,t) \to U_{\ka}$ as $t \to \infty$.
\end{thm}
In particular it follows from Theorem \ref{3.1.thm} that the solution corresponding to $u_0(x) = \sin(x)$ should converge to a nontrivial odd steady state. However, quite surprisingly, the following numerical computation shows that this is not always the case.

\subsection{A prototypical example on loss of parity}
We solve the following Allen-Cahn dynamics
\begin{align}\label{eq:example}
\begin{cases}
\partial_t u = \ka^2 \partial_{xx} u  - f(u),  \\
u_0 = \sin(x),
\end{cases}
\end{align}
using the first-order implicit-explicit scheme
\begin{equation}\label{sch:imex}
\frac{u^{n+1}-u^n}{\tau}  = \ka^2 \partial_{xx} u^{n+1} - f(u^n),\quad n\geq 0,
\end{equation}
with $\ka = 0.9$ and time step $\tau = 0.01$.
More precisely, we use the pseudo-spectral method (see, e.g., \cite{shen2011,trefethen00} with $N=256$ Fourier modes for space discretization.
The computed steady state and its energy evolution is illustrated in Figure \ref{fig0}.
It is observed that the oddness is lost gradually in time and the computed steady state is completely
incorrect.

\begin{figure}[!h]
\centering
\includegraphics[width=0.49\textwidth]{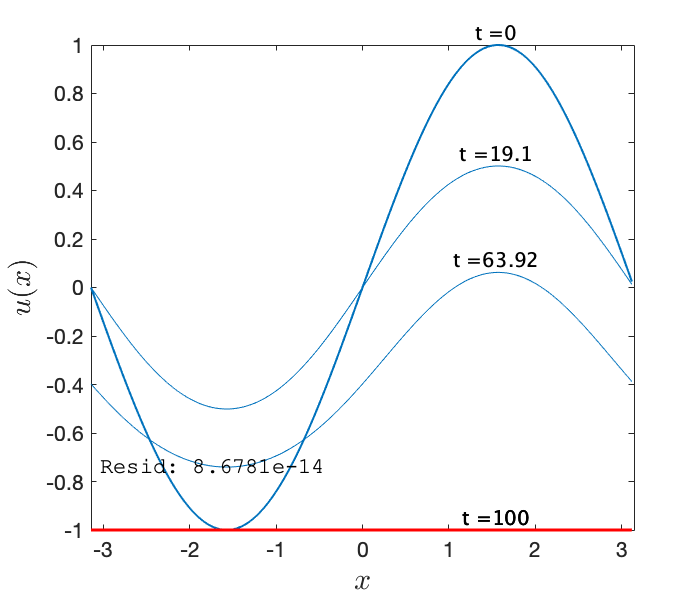}
\caption{\small Wrong steady state (red curves) for \eqref{eq:example}
computed by the first-order IMEX scheme with $\ka = 0.9$, $\tau = 0.01$, and $N = 256$.}\label{fig0}
\end{figure}

\subsection{Amplification of machine error}
We suspect that the above unreasonable symmetry-breaking phenomenon is connected with  the machine truncation error.
To understand this, one can consider a simple testing ODE
\begin{equation}\label{eq:testode}
u'(t) = u(t), \quad u_0 = 10^{-15}\approx 0,
\end{equation}
where $u_0\ll 1$ is regarded as the small machine error.
Clearly, the solution is $u(t) = u_0\exp(t)$ so that $u(35)\approx 1.5860 >1$.
However, if $u_0 = 0$ rigorously, $u(t)\equiv 0$.
This means that the initial machine error is amplified significantly after $t>35$.

A further investigation shows that the data stored by computer can be easily ``polluted'', even at time $t=0$.
For example, if we input $u_0 (x)= \sin(8x)$, we have
 $u_0(\frac{\pi j}{8})=0$ for all integers $j$.
However, standard software such as {\sc MATLAB}  gives  nonzero values due to machine errors.
As a consequence, the oddness is already violated at time $t=0$.
See Figure \ref{fig:intro6} for a graphical illustration.
This suggests that the values of $u$ may be contaminated by machine round-off errors at the very beginning of simulation.
\begin{figure}[!h]
\centering
\includegraphics[width=0.49\textwidth]{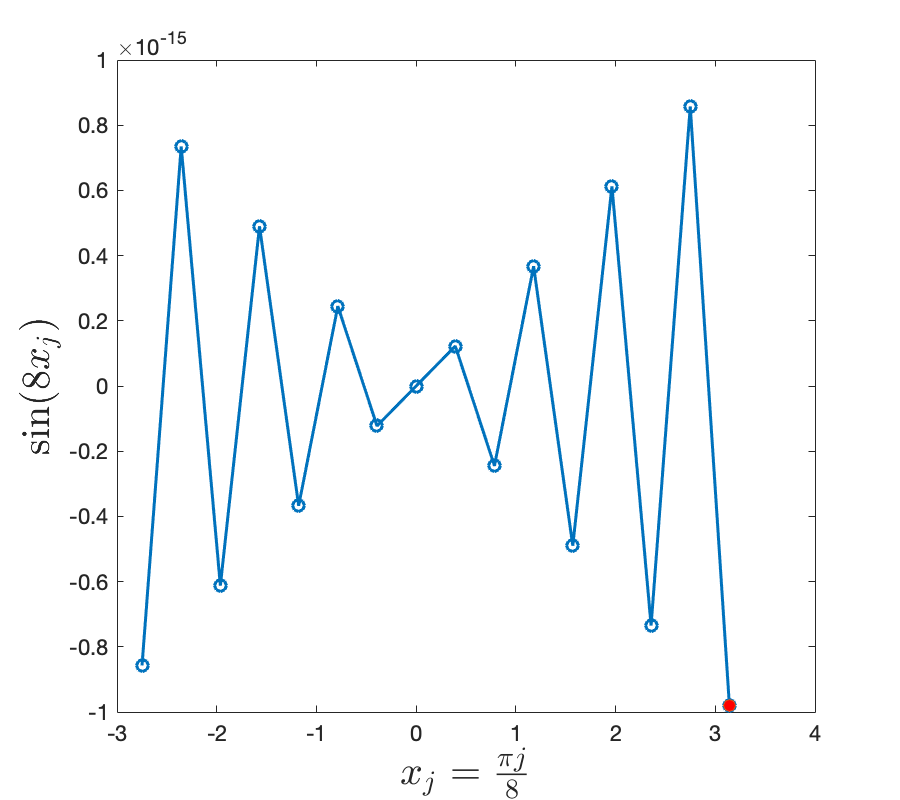}
\includegraphics[width=0.49\textwidth]{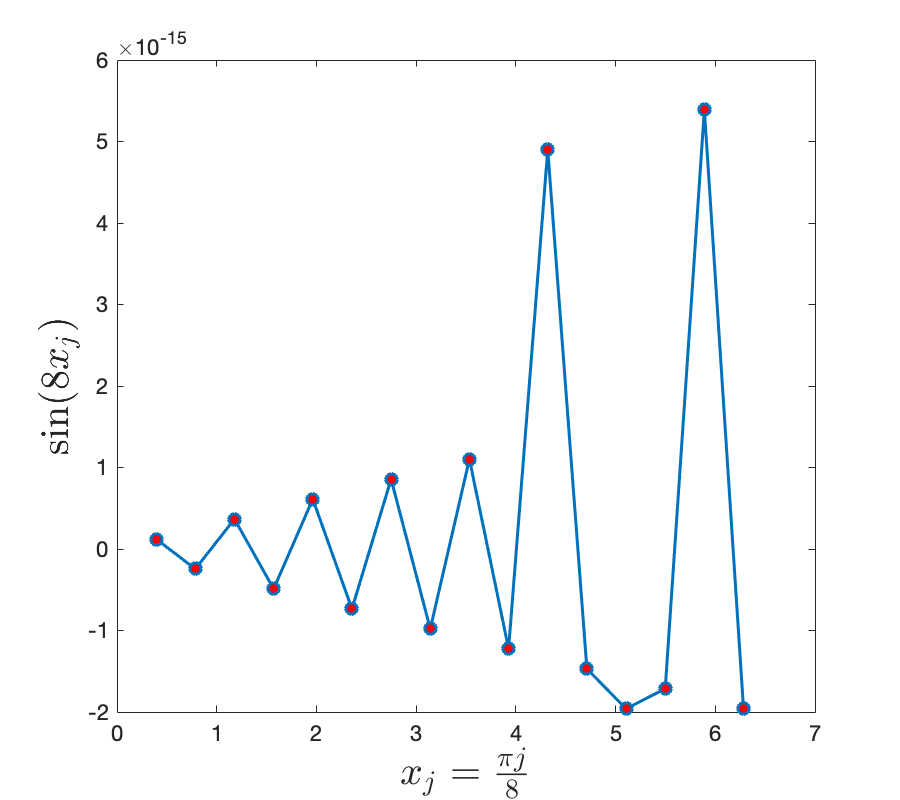}
\caption{\small Initial data of $u_0 = \sin(8x)$ at zero points respectively on $[-\pi,\pi]$ and $[0,2\pi]$, computed by {\sc MATLAB}. Symmetries at $(0,0)$ (left) and $(\pi,0)$ (right) are not preserved due to machine errors.}\label{fig:intro6}
\end{figure}

Even if we rectify the initial data to be odd (by numerically forcing $u_0$ to be
odd), the oddness can not be ensured after a few iterations, due to the gradual accumulation of machine error.
This issue is serious, since it leads to convergence to the incorrect steady state and thereby destroys the fidelity of the numerical simulation in the longer run.
In a deeper way the problem of the spurious growth of $u(0,t)$ (see Figure \ref{fig:intro4})
 is connected with the amplification of  projections into the exponentially-fast directions (such
 as projections into the steady state functions $\pm 1$ for the Allen-Cahn case) which
 can be summarized in the following abstract diagram in Figure \ref{tikz}.

\begin{figure}[!h]
\centering
\includegraphics[width=0.48\textwidth]{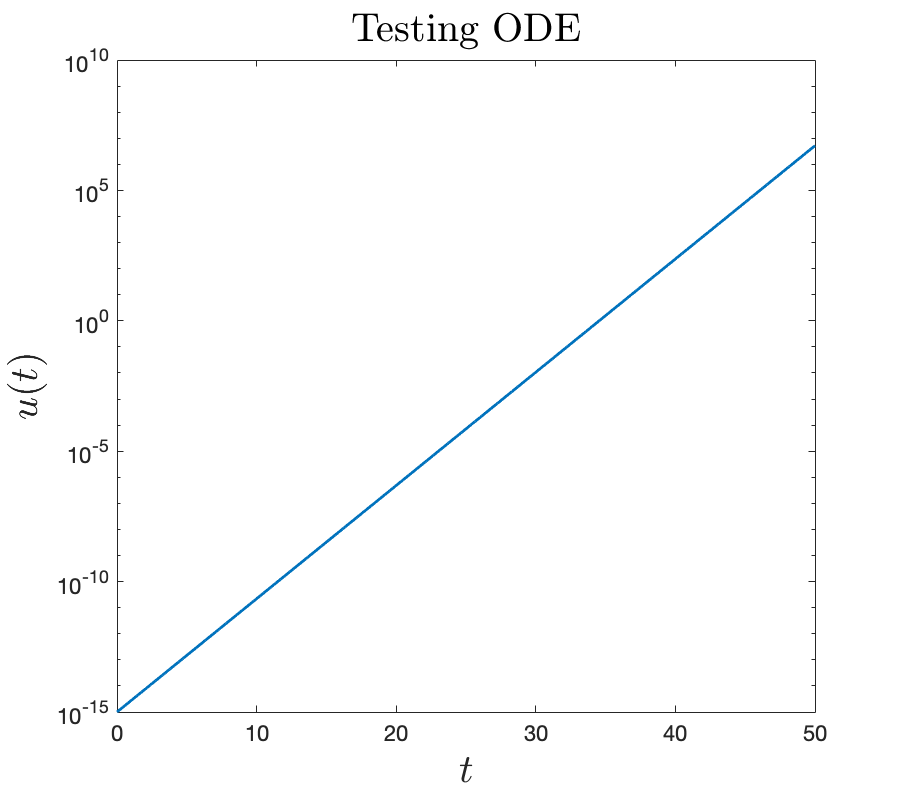}
\includegraphics[width=0.48\textwidth]{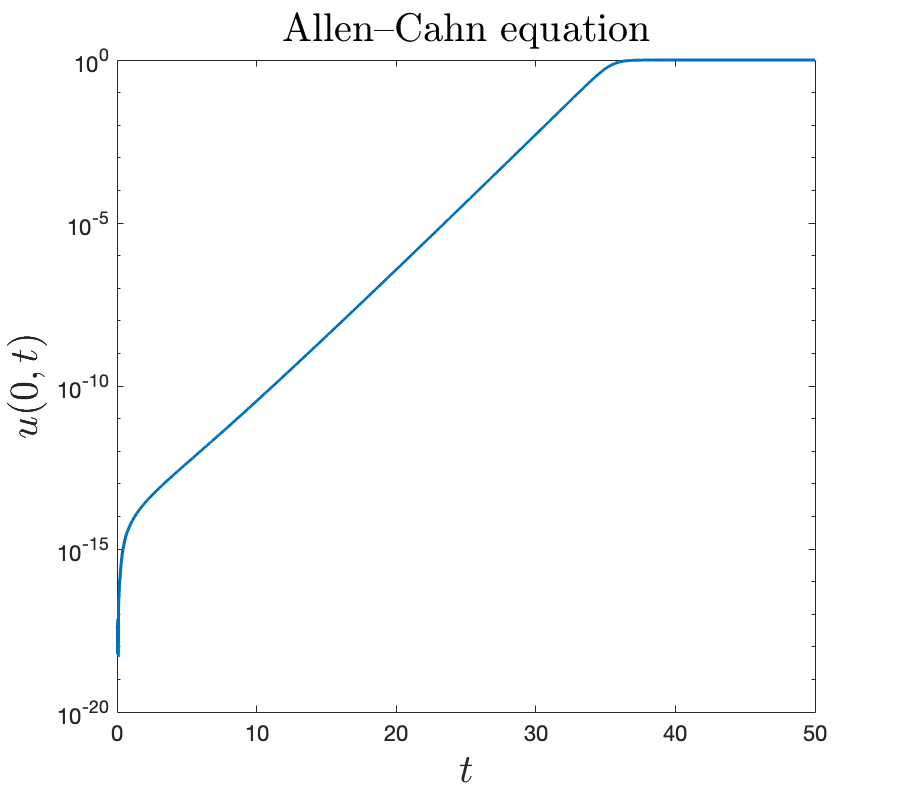}
\caption{\small Amplification of machine error of the testing ODE \eqref{eq:testode} (left); an example of the evolution of $u(0,t)$ for the Allen-Cahn equation \eqref{1.ac} with odd initial condition (right). }\label{fig:intro4}
\end{figure}

\tikzset{
block/.style={
    rectangle,
    draw,
    text width=8em,
    text centered,
    rounded corners
}
}
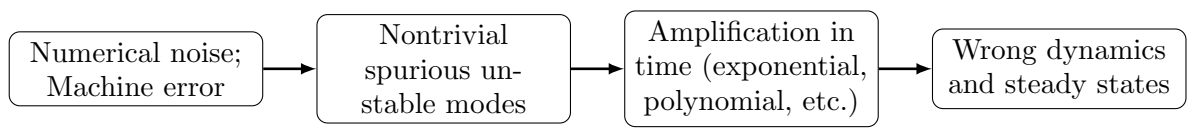
\begin{figure}[H]
{\small
\centering
\begin{tikzpicture}
\matrix (m)[matrix of nodes, column  sep=.7cm,row  sep=8mm, align=center, nodes={rectangle,draw, anchor=center} ]{
    |[block]| {Numerical noise; Machine error}    &
    |[block]| {Nontrivial spurious unstable modes}     &
    |[block]| {Amplification in time (exponential, polynomial, etc.) }  &
    |[block]| {Wrong dynamics and steady states}          \\
};
\path [>=latex,->,line width=1.pt] (m-1-1) edge (m-1-2);
\path [>=latex,->,line width=1.pt] (m-1-2) edge (m-1-3);
\path [>=latex,->,line width=1.pt] (m-1-3) edge (m-1-4);
\end{tikzpicture}
\caption{Schematic diagram of the amplification of machine error.}\label{tikz}
}
\end{figure}

\section{Our filtered approach}

\subsection{A new symmetry-preserving filter}
 It is known that $u(x, t)$ should preserve the oddness for all $t$ if the initial condition $u_0$ is odd.
However, standard numerical methods may converge to the spurious steady states  in long time due to the accumulation of machine error. 
To resolve this issue we introduce a new
symmetry-preserving filter which amounts to forcing the parity of the solution at every
iteration step.
Assume the initial data $u_0$ is odd and  has  $2\pi$ as its minimal period.
To preserve the oddness and the invariance of Fourier sine series, we propose a new Fourier filter imposing the zeroth Fourier coefficient to be zero and the real part of all Fourier coefficients to remain zero, namely, $\forall n\ge 0$,
\begin{subequations}\label{e4.2new}
\begin{align}
\widehat {u^n} (k = 0) = 0; \label{e4.2new1}
\end{align}
\begin{align}
\operatorname{Re} ( \widehat{u^n}(k) ) =0. \label{e4.2new2}
\end{align}
\end{subequations}
One should note that in a typical {\sc MATLAB} implementation, for
\eqref{e4.2new} to hold, we just need to use the command
\begin{equation}
\widehat{u^n}(k) = \operatorname{Im} ( \widehat{u^n}(k) )*1i \label{e4.2new2}
\end{equation}
in each iteration.
Figure \ref{fig0-2} illustrates the steady state computed with this Fourier filter respectively, where $u_0=\sin(x),\ka =0.9, ~T = 100$.
Compared to the result without filter in Figure \ref{fig0}, the proposed filter clearly helps the algorithm to converge to the correct steady state in long time.

Further, we illustrate plots of $\max_x|u_\infty|$ w.r.t. $\ka$ in Figure \ref{fig:umax}, computed respectively by the IMEX Fourier scheme \eqref{sch:imex} without and with filter.
It can be seen that for the classical method without filter, $\max_x|u_\infty| = 1$ when $\ka\gtrsim 0.70$, indicating that the computed steady state is wrong.
On the other hand,
the filtered method  guarantees that the numerical solution converges to the correct steady state.

\begin{figure}[!h]
\centering
\includegraphics[width=0.49\textwidth]{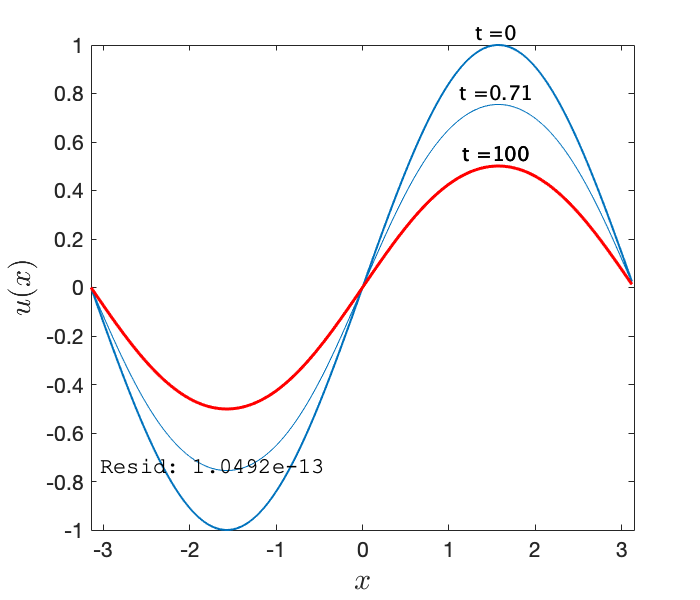}
\caption{\small Correct steady state (red curves) for problem \eqref{eq:example}
computed by the filtered first-order IMEX scheme with $\ka = 0.9$, $\tau = 0.01$, and $N = 256$.}\label{fig0-2}
\end{figure}

\begin{figure}[!h]
\centering
\includegraphics[width=0.49\textwidth]{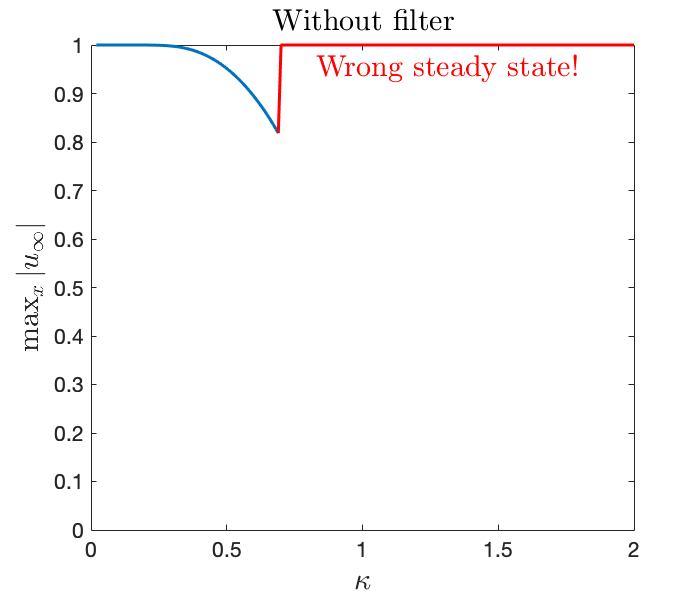}
\includegraphics[width=0.49\textwidth]{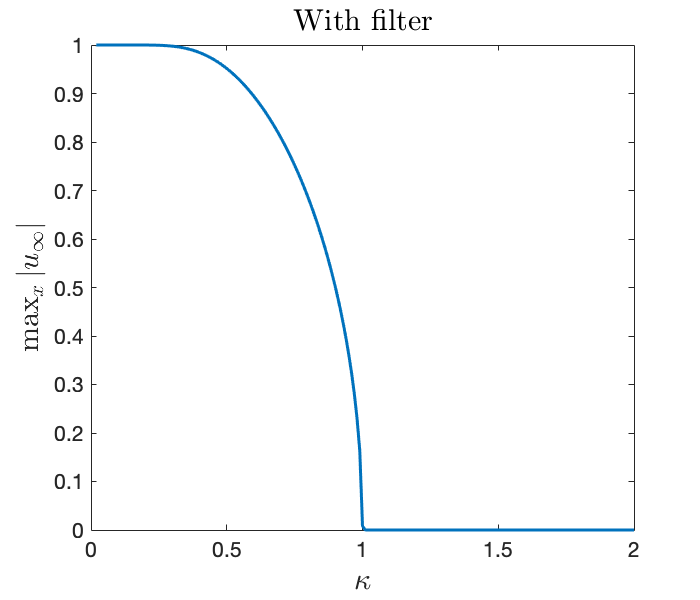}
\caption{\small $\max_{x}|u_\infty|$ w.r.t. $\ka$ computed by the first-order IMEX scheme respectively  without filter (left) and with filter (right).}\label{fig:umax}
\end{figure}

\subsection{Convergence of the filtered approach}
In this part, we give some mathematical analysis to demonstrate that the filtered numerical solution can converge to the desired steady state solution.

To explain the heart of the matter, we shall carry out the analysis for a semi-discrete model with perturbations introduced by machine error.  To simplify the analysis we only consider a restricted
set of parameters. For example, we take the representative case $\ka = 0.9$ which was considered in the earlier numerical experiments (see Figures \ref{fig0} and \ref{fig0-2}) to ensure the uniqueness of the zero-up
ground state and simplify the relevant perturbative analysis.
To minimize technicality we did not work with the optimal assumptions. In forthcoming works
we will develop a comprehensive theoretic framework to address all issues related to full
discretization, truncation error and so on.

To this end, we consider $\ka = 0.9$ and
\begin{equation}
\frac{w-u}{\tau} = \ka^2 \partial_{xx} w - (u^3-u), \quad\mbox{on } [-\pi,\pi],
\end{equation}
where $u,w:\mathbb T\rightarrow \mathbb R$ are both odd and smooth.
Clearly
\begin{equation}
w = (1-\ka^2 \tau\partial_{xx})^{-1}[u-\tau(u^3-u)].
\end{equation}
We denote the map $u\mapsto w$ as
\begin{equation}
w =\mathcal N(u).
\end{equation}
The main model for the filtered solution is given by
\begin{equation}\label{fil1}
\left\{
\begin{aligned}
u^{n+1} & = \mathcal N(v^n),\\
v^{n+1} & = u^{n+1}+\eps^{n+1},\\
v^0 & = u^0 + \eps^0,
\end{aligned}
\right.
\end{equation}
where $\eps^n:\mathbb T\rightarrow \mathbb R$ are given machine error functions.
Our standing assumption is that
\begin{equation}\label{fil2}
\sup_{n\geq 0} \|\eps^n\|_{H^1(\mathbb T)}\leq \eps_* \ll 1,
\end{equation}
where $\eps_*>0$ is a very small constant which corresponds to machine precision.
We shall assume that all $\eps^n$ are odd.
This very assumption corresponds to our filtering technique, i.e., we apply the Fourier filter at each iteration step. One should note that $u^{n+1}$ corresponds to the exact numerical
solution computed flawlessly (i.e. without machine error) using the numerical scheme, and
$v^{n+1}$ corresponds to the actual numerical solution taking into account of the
machine error.

We point out that the $H^1$ assumption \eqref{fil2} is reasonable at least
	for moderately small $N$, where $N$ is the number of Fourier modes in
	the Fourier collocation method.  A more realistic assumption is to use
	only $L^{\infty}$-norm. Here to simplify the relevant analysis we employ
	this slightly stronger assumption.

In Lemma \ref{lemfil3} below, we just consider the case of $\kappa=0.9$, which ensures that there exists
a unique odd zero-up ground state. Moreover, one should note that as $\kappa \to 1$, the energy
of the odd zero-up ground state goes to $\frac \pi 2$ (see Figure \ref{fig:Eeps}) and
the constant $C_1$ in \eqref{C1eee0} has to be taken sufficiently close to zero. In order not
to overburden the reader with these technicalities and to simplify the analysis, we consider the
simplest case of $\kappa=0.9$ to illustrate the main ideas.

\begin{lem}[Almost steady state]\label{lemfil3}
	Fix $\ka = 0.9$ and assume that $U_{\operatorname{steady}}$ is the unique odd zero-up ground state for the equation $\kappa^2u''+u-u^3=0$ on $\mathbb T=[-\pi,\pi]$. Suppose $u:\mathbb T\to\mathbb{R}$ is a smooth odd function
	satisfying
	\begin{align}\kappa^2u''+u-u^3=r(x),\end{align}	
	where the residual error function $r(x)$ satisfies
	\begin{equation}
		\|r\|_2\le r_*\ll1.
	\end{equation}	
	Assume
	\begin{align} \label{C1eee0}
	E(u)=\int_{\mathbb T}\left(\frac12\kappa^2(u')^2+\frac14(u^2-1)^2\right)dx\leq\frac\pi2-C_1,\quad C_1=0.001.
	\end{align}	
	Suppose $u'(0)\ge 0$, then
	\begin{align}
	\|u-U_{\operatorname{steady}}\|_{H^1(\mathbb T)}\leq\alpha_1r_*,
	\end{align}
	where $\alpha_1 = O(1)$ is an absolute constant.
\end{lem}

\begin{proof}
	We only sketch the details. Denote $u'(0)\ge 0$ and observe that we have $\|u'\|_\infty=O(1)$. Consider
	\begin{equation} \label{4.12Ad0}
		\begin{cases}
			\ka^2u''+u-u^3=r,\\
			u(0)=0,\quad u^{\prime}(0)=p,
		\end{cases}	
		\quad
		\begin{cases}
			\ka^2w''+w-w^3=0,\\
			w(0)=0,\quad w^{\prime}(0)=p.
		\end{cases}		
	\end{equation}
	Note that here we use the letter $p$ to denote $u^{\prime}(0)$ which is given (i.e. $p$ is not
	a parameter). 	
	Since $u$ and $w$ have the same initial conditions, we obtain that
	\begin{align}\|u-w\|_{H^1(\mathbb T)}\leq O(1)r_*,\quad |E(u)-E(w)|\ll1.\end{align}	
	Clearly $E(w)\le\frac\pi2-0.001.$ By using the classification of steady state \cite{LQTW_2}, we obtain that $w$ must coincide with $U_{\operatorname{steady}}$. In particular, this shows that
	 $p=u^{\prime}(0)$ in
	\eqref{4.12Ad0} cannot be arbitrary and has to coincide with the
	ground state. The desired result follows easily.	
\end{proof}

\begin{lem}\label{lem-stability}
Fix $\ka=0.9$ and assume $U_{\mathrm{steady}}$ is the unique odd zero-up ground state for the equation
$$\ka^2u''+u-u^3=0\quad \mbox{on}\quad \mathbb T=[-\pi,\pi].$$
Suppose that $u\in H^1(\mathbb T)$ is an odd function on $\mathbb{T}$ with
	\begin{align}
	E(u)=\int_{\mathbb T}\left(\frac12\kappa^2(u')^2+\frac14(u^2-1)^2\right)dx\le\frac\pi2-C_1,\quad C_1=0.001.
	\end{align}	
Then we have
\begin{equation}
\label{4.com}
\min\{\|u-U_{\mathrm{steady}}\|_{H^1(\mathbb{T})},~
\|u+U_{\mathrm{steady}}\|_{H^1(\mathbb{T})}\}
\leq C_2\sqrt{E(u)-E(U_{\operatorname{steady}})},
\end{equation}	
where $C_2>0$ is an absolute constant.
\end{lem}

\begin{proof}
This follows from a general (nontrivial) spectral estimate established in \cite{LQTW_2}.
\end{proof}

With the above two lemmas, we now present the main theorem of this section.

\begin{thm}
	Assume the initial state $u^0:\mathbb T\to\mathbb R$ is odd, $\|u^0\|_\infty\le 1$ and
	\begin{align} \label{Eu0.002}
	E(u^0)<\frac\pi2-0.002.\end{align}	
	Suppose that the machine error functions $\{\eps^n\}_{n\geq0}$ are odd and satisfy \eqref{fil2}. Let $\ka = 0.9$ and consider \eqref{fil1} with $\tau_*\le\tau\leq\frac1{2}$, where $\tau_*= \sqrt{\eps_*}$. Then the following hold.
	\begin{enumerate}
		\item [(1)] $E(U_{\operatorname{steady}})\leq E(v^{n+1})\leq\frac\pi2-0.001,$~ $\forall\,n\geq0.$
		\item [(2)] For $n$ sufficiently large, we have
		\begin{align}
		\label{4.neb}
		\min\left\{\|v^{n}-U_{\operatorname{steady}}\|_{H^1(\mathbb T)},~
		\|v^{n}+U_{\operatorname{steady}}\|_{H^1(\mathbb T)}\right\}\leq\tilde\alpha\eps_*^{\frac18},\end{align}
		where $\tilde \alpha = O(1)$ is an absolute constant.
	\end{enumerate}
\end{thm}

\begin{proof}
By a recent result in
\cite{Li2021}, for any $\tau_*\le \tau \le \frac 12$, we have
\begin{align*}
\sup_{n\ge 0} \max\{\|u^n\|_{\infty}, \;\|v^n\|_{\infty}\} \le 1.1.
\end{align*}
	We begin with the basic energy inequality. Namely for any $\tau_*\leq \tau\leq\frac{1}{2}$ and $w=\mathcal{N}(u)$ with $u$ odd, $E(u)\leq\frac\pi2-0.001$ and $\max\{\|w\|_\infty,\|v\|_\infty\}\le 1.1$, we have
	\begin{equation}
		\label{fil6}
		\frac{1}{5\tau}\|w-u\|_2^2+\frac12\ka^2\|\partial_x(w-u)\|_2^2+E(w)
		\leq E(u).
	\end{equation}	
	The main inductive assumption is 	
\begin{equation} \label{4e20}
     E(v^n)\leq\frac\pi2-0.001.
\end{equation}
	Clearly, this holds at the base step $n=0$. We now assume it holds at step $n$ and consider the solution at $n+1$ given by
	\begin{align}v^{n+1}=u^{n+1}+\eps^{n+1}=\mathcal N(v^n)+\eps^{n+1}.\end{align}
	Now we discuss two cases.
	
\vskip .25cm

	\noindent Case 1: ${\|u^{n+1}-v^{n}\|_2}/{\tau}\leq\delta_*^{(1)},$ where $\delta_*^{(1)}\ll1$ is some suitably small constant to be specified at the end of the proof. Note that by using $u^{n+1}=\mathcal{N}(v^n)$ and the fact that $E(v^n)\leq\frac\pi2-0.001$, we have
	\begin{align}E(u^{n+1})\leq E(v^n)\quad\mathrm{and}\quad \|u^{n+1}\|_{H^1(\mathbb T)}=O(1).\end{align}
	Thus
	\begin{equation}
		\ka^2\partial_{xx}u^{n+1}-f(u^{n+1})=\frac{u^{n+1}-v^n}{\tau}+f(v^n)-f(u^{n+1})
		=r^{n+1}(x),	
	\end{equation}
	where $\|r^{n+1}\|_2\ll1 $. By Lemma \ref{lemfil3} we obtain
	\begin{equation}
		\min\left\{\|u^{n+1}-U_{\operatorname{steady}}\|_{H^1(\mathbb T)},~
		\|u^{n+1}+U_{\operatorname{steady}}\|_{H^1(\mathbb T)}\right\}\leq\alpha\delta_*^{(1)},
	\end{equation}
	where $\alpha>0$ is an absolute constant (note that in the more general case, $\alpha$ should depend on $\ka$). Thus in this case we have
	\begin{equation}
	\label{4.31}
		\min\left\{\|v^{n+1}-U_{\operatorname{steady}}\|_{H^1(\mathbb T)},~
		\|v^{n+1}+U_{\operatorname{steady}}\|_{H^1(\mathbb T)}\right\}\leq \alpha\delta_*^{(1)}+\eps_*.
	\end{equation}
	Clearly then $E(v^{n+1})\leq\frac\pi2-0.001$.
	
\vskip .25cm
	\noindent Case 2: ${\|u^{n+1}-v^{n}\|_2}/{\tau}>\delta_*^{(1)}.$ By Eq. \eqref{fil6} we have
	\begin{eqnarray}\label{4.27-0}
			E(v^n)
			&\geq& E(u^{n+1})+\frac{1}{5\tau}\|u^{n+1}-v^n\|_2^2\nonumber\\
			&\geq&  E(u^{n+1})+\frac{1}{10\tau}(\delta_*^{(1)})^2\tau^2 +\frac{1}{10\tau}\|u^{n+1}-v^n\|_2^2\nonumber \\
			&\geq&	E(v^{n+1})+O(\eps_*)+\frac{1}{10}(\delta_*^{(1)})^2\tau +\frac{1}{10\tau}\|u^{n+1}-v^n\|_2^2.
	\end{eqnarray}
	Thus if
	\begin{align}\label{4.28-0}
	O(\eps_*)+\frac{1}{10}(\delta_*^{(1)})^2\tau>0,
	\end{align}
	then we obtain
	\begin{align}\label{4.29-0}
	E(v^n)\geq E(v^{n+1})+\frac{1}{10\tau}\|u^{n+1}-v^n\|_2^2.\end{align}
	Thus for all cases we have proved (\ref{4e20}).
	With \eqref{4e20} in hand, we now claim that there exists some $n_0\ge 1$ such that Case 1 holds at $n=n_0$, i.e.
	\begin{align}\label{4.31-0}
	\frac{\|u^{n+1}-v^{n}\|_2}{\tau}\leq\delta_*^{(1)}.
	\end{align}
	Assume this is not true, then we are always in Case 2.
	By \eqref{4.29-0} we obtain
	\[
	\frac{1}{10\tau} \sum_{n=1}^\infty \|u^{n+1}-v^n\|_2^2<\infty.
	\]
	This clearly implies the existence of $n_0$. Thus the claim \eqref{4.31-0} holds.
	
	Next, we show that if for some $n_0$, \eqref{4.31-0} holds, then the energy of all future $v^{n+1},~n\ge n_0$ will remain in the vicinity of the ground state, i.e.
	\begin{align}\label{4.33-0}
	E(v^{n+1})\leq E(U_{\mathrm{steady}}) + C_3\cdot (\alpha \delta_*^{(1)}+ \eps_*),
	\end{align}
	where $C_3>0$ is an absolute constant.
	Indeed, \eqref{4.33-0} holds for $n=n_0$ and any $m>n_0$ which is in Case 1. Now for any $m>n_0$ such that $m$ is in Case 2.
	We define
	\[
	m_* = \min\{j: j\le m \mbox{ and any $\widetilde m$ in $[j,m]$ is in Case 2}\}.
	\]
Clearly by monotonicity of energy in Case 2, we have
	\begin{align}
	E(v^{n+1})\le E(v^{m_*+1})\le E(v^{m_*}).
	\end{align}
	Since by definition, $m_*-1$ is in Case 1, we have
	\begin{align}
	E(v^{m_*})\leq E(U_{\mathrm{steady}}) + C_3\cdot (\alpha \delta_*^{(1)}+ \eps_*).
	\end{align}
	thus \eqref{4.33-0} holds for all $v^{n+1}$, $n\geq n_0$.
	It follows from Lemma \ref{lem-stability} that for any $n\ge n_0$,
	\begin{equation}\label{4.37-0}
\min\left\{\|v^{n+1}-U_{\mathrm{steady}}\|_{H^1(\mathbb{T})},~
\|v^{n+1}+U_{\mathrm{steady}}\|_{H^1(\mathbb{T})}\right\}
\leq C_2\sqrt{C_3\cdot (\alpha \delta_*^{(1)}+ \eps_*)}.
	\end{equation}	
	
	To see that \eqref{4.37-0} gives us the desired result, we now specify the parameters used in the proof.
	We take $\tau_* = \eps_*^{\frac 1 2}$ and $\delta_*^{(1)} = \beta\eps_*^{\frac 1 4}$ with $\beta$ being an $O(1)$ constant, then \eqref{4.28-0} is clearly satisfied.
	Plugging the values of $\delta_*^{(1)}$ and $\eps_*$ into \eqref{4.37-0}, we obtain the desired conclusion.
\end{proof}

We close this section by making several remarks relevant to the above theorem and its proof.

\begin{itemize}
\item
The condition \eqref{Eu0.002} is to ensure that the energy of $u^0$ is below
$\pi/2=E(0)$ where $E(0)$ is the energy of $u\equiv 0$.  This simple assumption avoids
the scenario that the solution becomes zero in finite time.
For $u^0 = \sin(x)$ and $\ka =0.9$, we can compute
\begin{align}E(u^0) \approx 1.5327<\frac \pi 2 -0.002.\end{align}
\item
As we have mentioned earlier, the reasons why we take $\kappa=0.9$ are:
1) to ensure the uniqueness of the odd zero-up ground state; 2) to simplify the perturbative
analysis around the ground state;
3) to back up the numerical experiments (see Figure \ref{fig0} and \ref{fig0-2}). In forthcoming works, we shall treat more general
cases.
\item
To simplify the analysis, we did not optimize the parameters such as $\tau_*$ and the other relevant parameters used in the proof.
In view of the smallness of the machine error $\eps_*$, the cut-off $\tau_*$ is quite a realistic assumption since we usually do not adopt exceedingly small time steps in such long time simulations.
Further fine-tuning is certainly  possible but we shall not dwell on this issue here.
\item
From our analysis, it is clear that an improved algorithm can include a stopping criterion when the residual error becomes suitably small. This would also bring drastic simplifications in our proof.
In particular, one does not need to consider the situation where the numerical solution hops between Case 1 and Case 2 intermittently.
\end{itemize}

\section{More numerical experiments}
In the following numerical experiments, the time step is $\tau = 0.01$ and the number of Fourier modes is $N = 256$ by default.
In addition, the ``numerical'' steady state is obtained when the residual error of the scheme
is smaller than $\mathtt {Tol} = 10^{-12}$ or the the time arrives at $t = 10^5$.

\subsection{More 1D Allen-Cahn examples}

\begin{example} \label{exam1}
{\em We consider more initial conditions.
Let the parameter $\ka = 0.1$ and the initial data $u_0 = \sin(x)$, $\frac 1 2\left(\sin(x)+\sin(2x)\right)$, $\frac 1 2\left(\sin(x)+\sin(4x)\right)$ or $\frac 1 2\left(\sin(x)+\sin(8x)\right)$, respectively.
}
\end{example}

The approximate steady states are computed using the filtered first-order IMEX pseudo-spectral method.
In Figure \ref{fig2}, it shows that for these odd initial conditions, $u(x,t)$ converges to the same steady state with period $2\pi$.

\begin{figure}[!h]
\centering
(a) \includegraphics[width=0.44\textwidth]{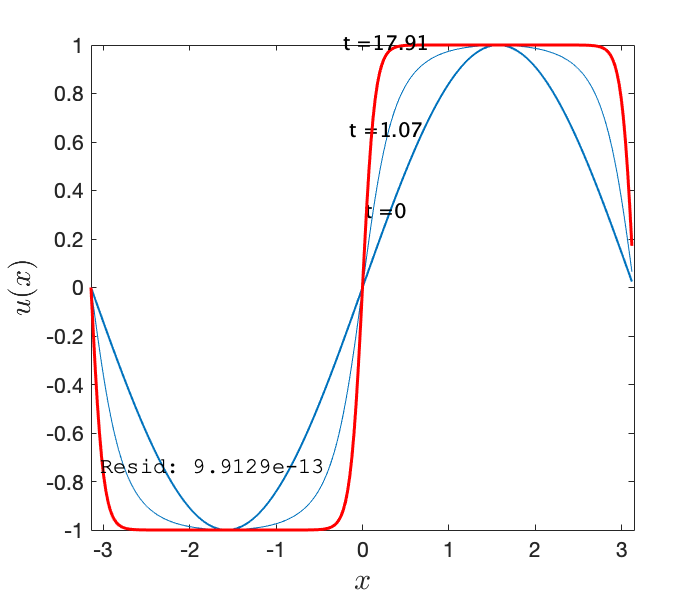}
(b) \includegraphics[width=0.44\textwidth]{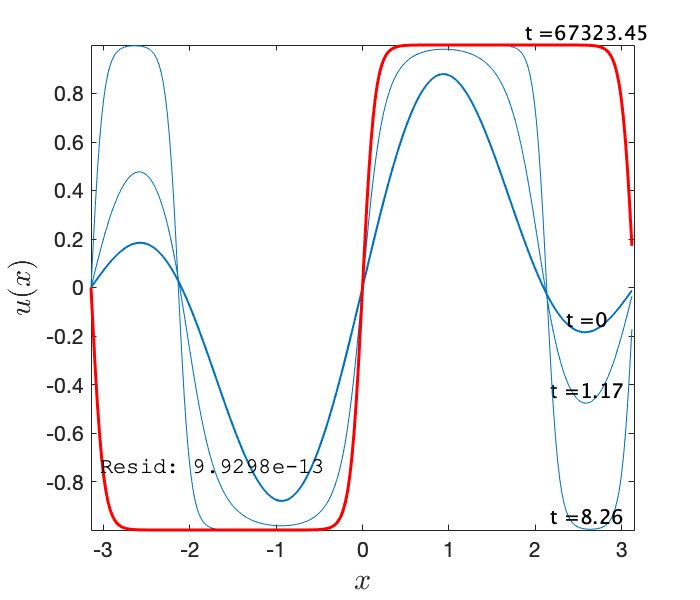}
(c) \includegraphics[width=0.44\textwidth]{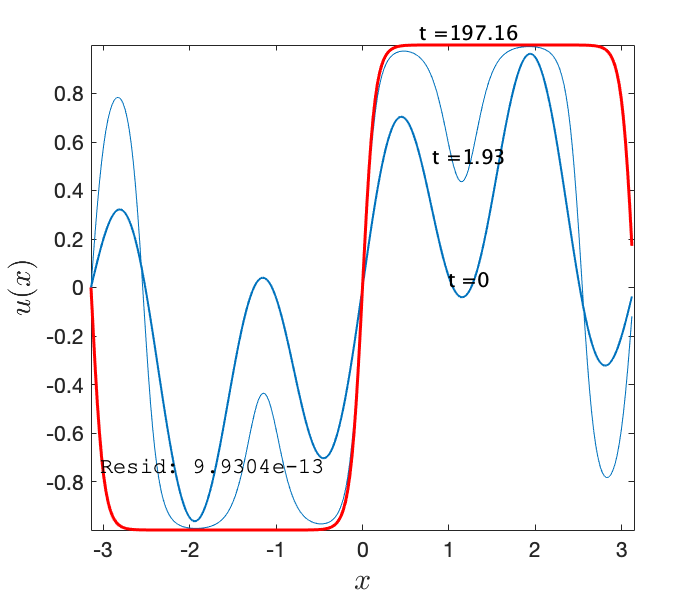}
(d) \includegraphics[width=0.44\textwidth]{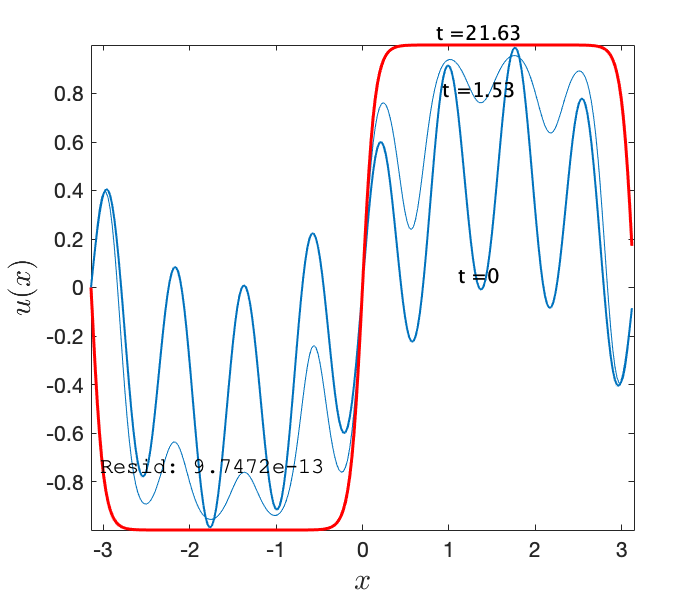}
\caption{\small Example \ref{exam1}: Numerical steady states (red curves) for different initial conditions (bold blue curves): (a): $u_0 = \sin(x)$, (b): $\frac 1 2\left(\sin(x)+\sin(2x)\right)$, (c): $\frac 1 2\left(\sin(x)+\sin(4x)\right)$ and (d): $\frac 1 2\left(\sin(x)+\sin(8x)\right)$.}\label{fig2}
\end{figure}

\begin{example} \label{exam2}
{\em  We consider more parameter of $\ka$.
Let $u_0 = \sin(x)$, $\ka = 0.999$ and $1.001$ respectively.
}
\end{example}

It is observed in Figure \ref{fig1} that when $\ka$ is slightly smaller than $1$, the steady state is nonzero, while when $\ka$ is slightly larger than $1$, the steady state is zero. This is in good agreement with our theory in Section 2.

\begin{figure}[!h]
\centering
(a) \includegraphics[width=0.43\textwidth]{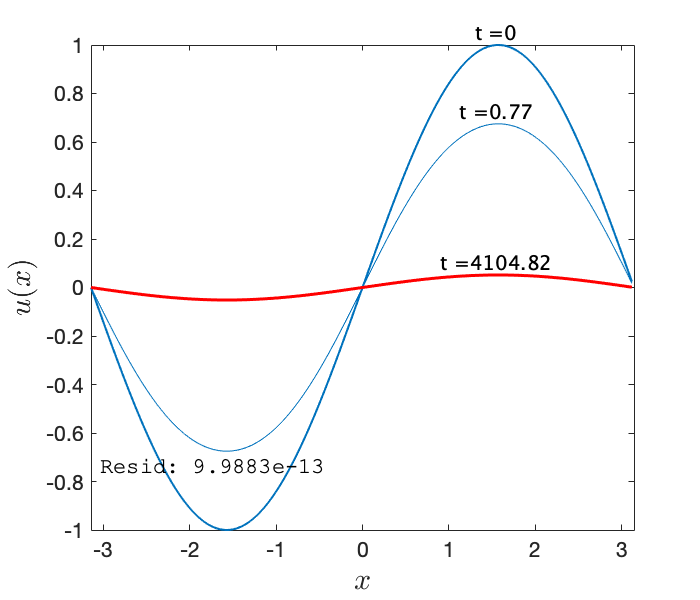}
(b) \includegraphics[width=0.43\textwidth]{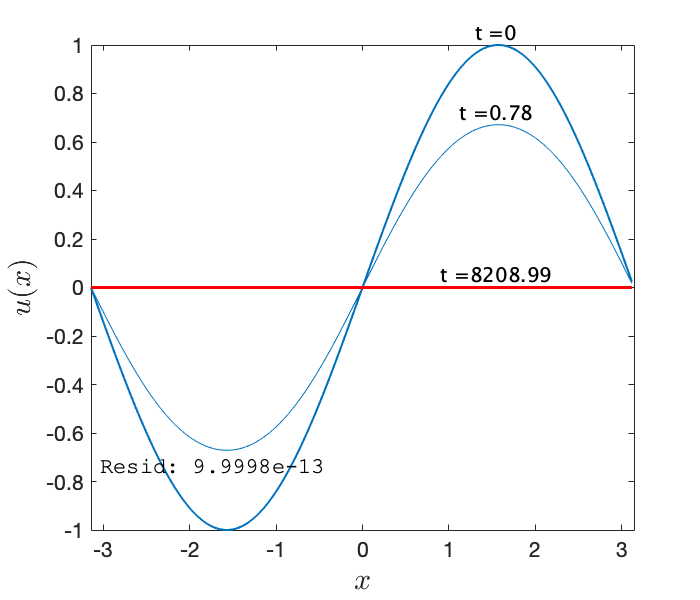}
\caption{\small Example \ref{exam2}: Computed steady states (red curves) respectively in the cases of (a): $\ka = 0.999$, and (b): $\ka = 1.001$. }\label{fig1}
\end{figure}



\begin{example} \label{exam3}
{\em We consider an intricate case of metastable states for the Allen-Cahn equation.
Let $\ka = \sqrt{0.001}$ and take two different initial functions $u_0 = \frac 1 2(\sin(x)+\sin(2x))$ and $u_0=\frac 12 (\sin(x)+\sin(8x))$.
}
\end{example}

It is observed numerically even if the oddness is preserved the numerical solution can get stuck in the metastable states.
In particular, it is seen from the left-hand side of Figure \ref{fig5_1} that for $u_0 = \frac 1 2(\sin(x)+\sin(2x))$, the corresponding solution 
 $u$ is stuck in the metastable state when the residual error reaches $\mathtt{Tol}=10^{-12}$.
Similarly, on the right-hand side for $u_0 =\frac12(\sin(x)+\sin(8x))$,  the corresponding solution  $u$ is stuck in another metastable state.
An interesting issue is to investigate these metastable behaviors.

\begin{figure}[!h]
\centering
(a) \includegraphics[width=0.43\textwidth]{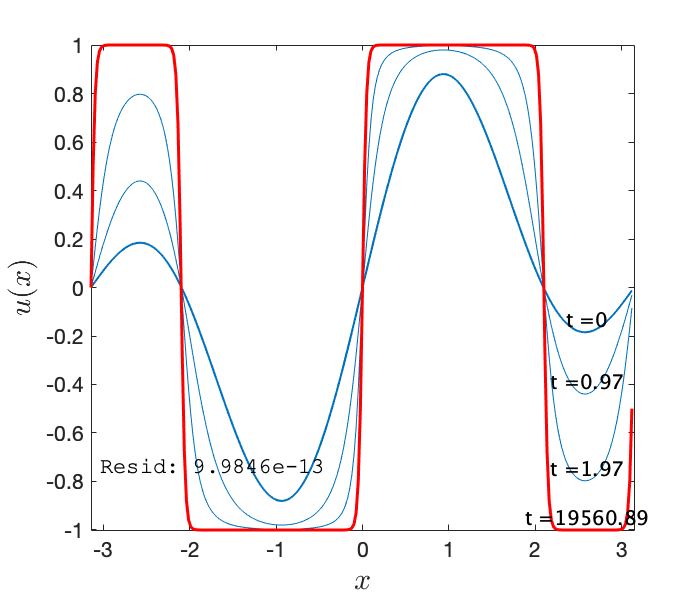}
(b) \includegraphics[width=0.43\textwidth]{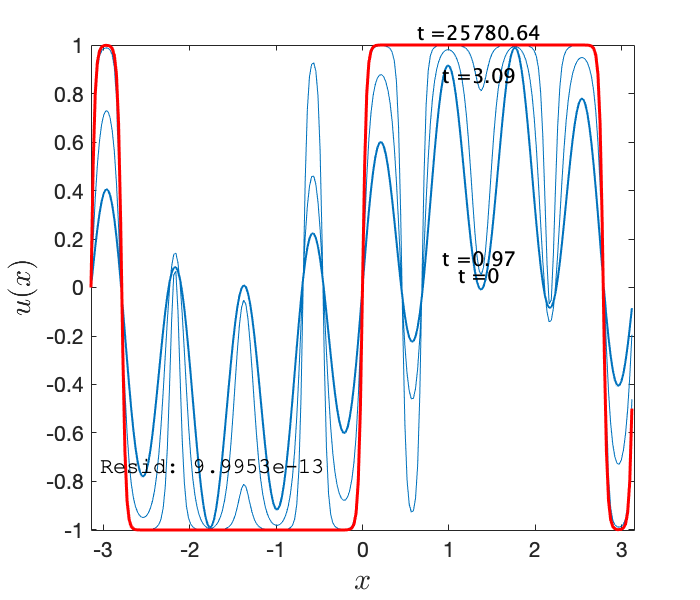}
\caption{\small Example \ref{exam3}: Metastable states (red curves) for (a): $u_0 = \frac 1 2(\sin(x)+\sin(2x))$ and (b): $u_0=\frac 1 2(\sin(x)+\sin(8x))$, with $\ka=\sqrt{0.001}$.}\label{fig5_1}
\end{figure}



\subsection{A 2D Allen-Cahn example}

\begin{example} \label{exam4}
{\em
We consider the 2D Allen-Cahn equation
\begin{align}
\partial_t u = \ka^2\Delta u + u-u^3\quad\mbox{on } \mathbb T^2 = [-\pi,\pi]^2,
\end{align}
with $\ka = 0.1$ and $u_0(x,y)= \sin (x)\sin(y)$.
}
\end{example}

Clearly, $u(x,y,t)$ should have the form
\begin{equation}
u(x,y,t) = \sum_{k_1,k_2\geq 1} c_{k_1,k_2}(t)\sin(k_1 x) \sin(k_2 y),
\end{equation}
where $c_{k_1,k_2}$ is the Fourier coefficient of the $(k_1^{\mathrm {th}},k_2^{\mathrm {th}})$ mode.
This indicates the symmetries $u(x,y,t) = -u(-x,y,t)$ and $u(x,y,t) = -u(x,-y,t)$.
In our computation, we use the first-order-IMEX pseudo-spectral method to solve the 2D Allen-Cahn equation with $\tau = 0.01$ and $N_x\times N_y = 256\times 256$ Fourier modes.

\begin{figure}[!h]
\centering
\includegraphics[width=0.33\textwidth]{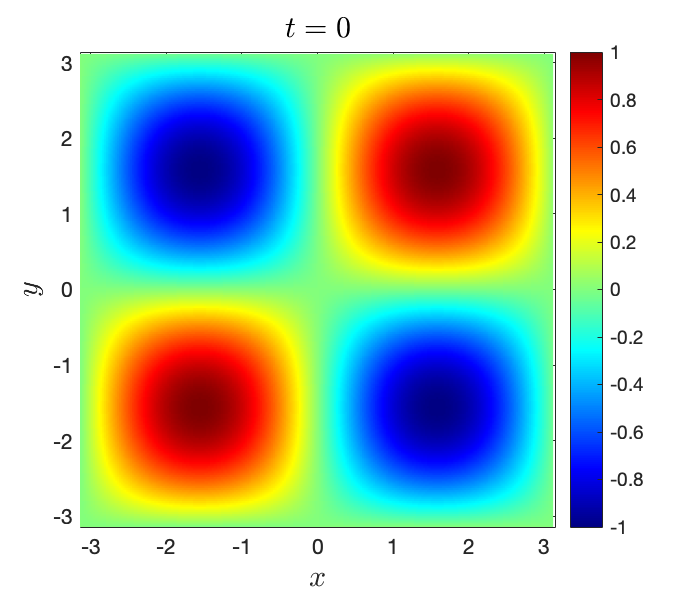}
\includegraphics[width=0.33\textwidth]{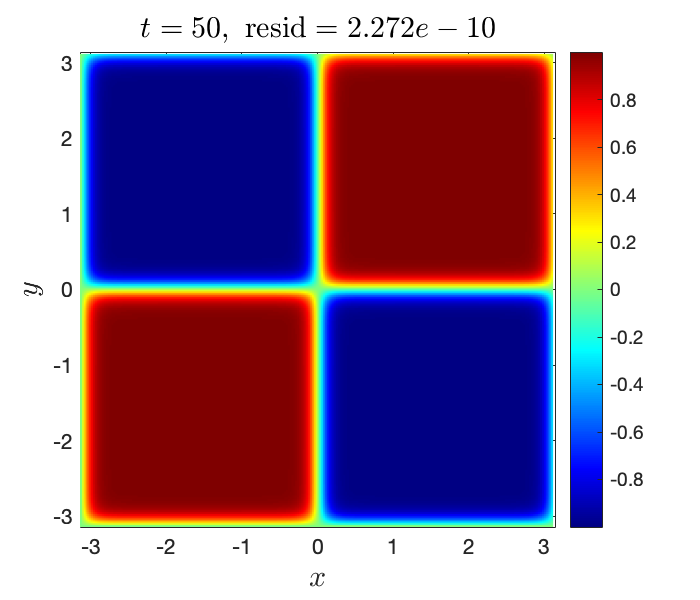}\\
\includegraphics[width=0.33\textwidth]{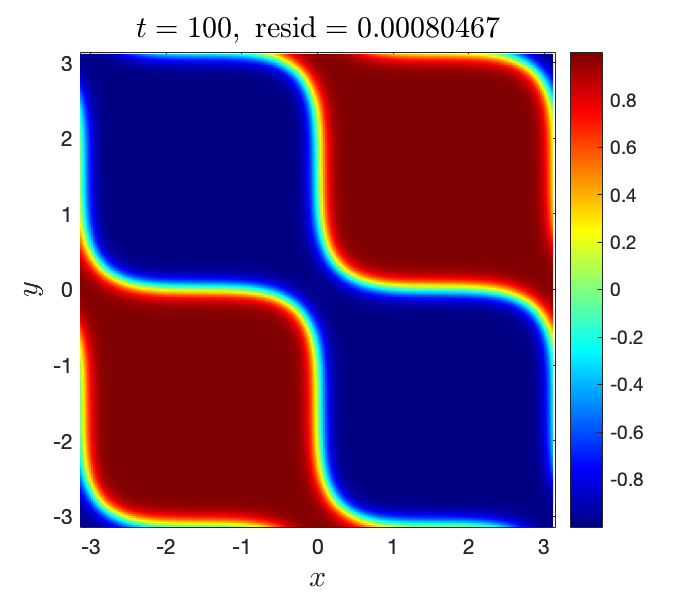}
\includegraphics[width=0.33\textwidth]{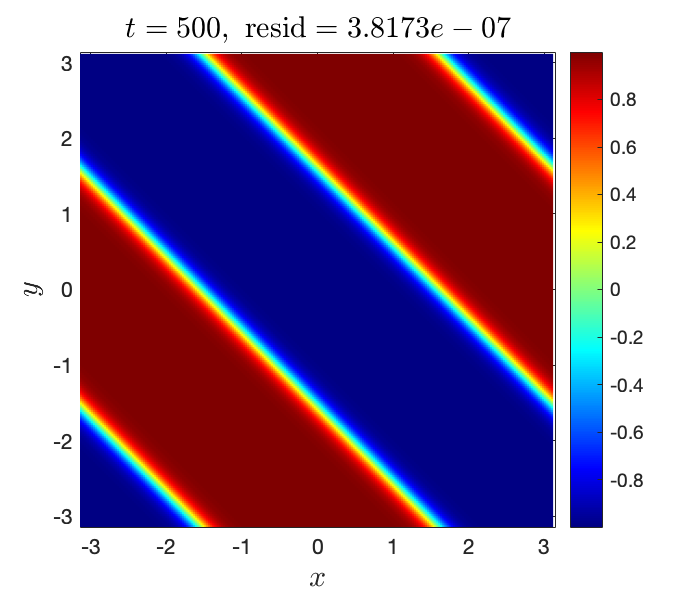}
\caption{\small Example \ref{exam4}: Dynamics of 2D Allen-Cahn equation using the first-order IMEX method without filter (symmetry breaking happens!) where $\ka = 0.1$, $u_0 =\sin(x)\sin(y)$, $\tau= 0.01,~N_x=N_y = 256$. }\label{fig:ac2d1}
\centering
\includegraphics[width=0.33\textwidth]{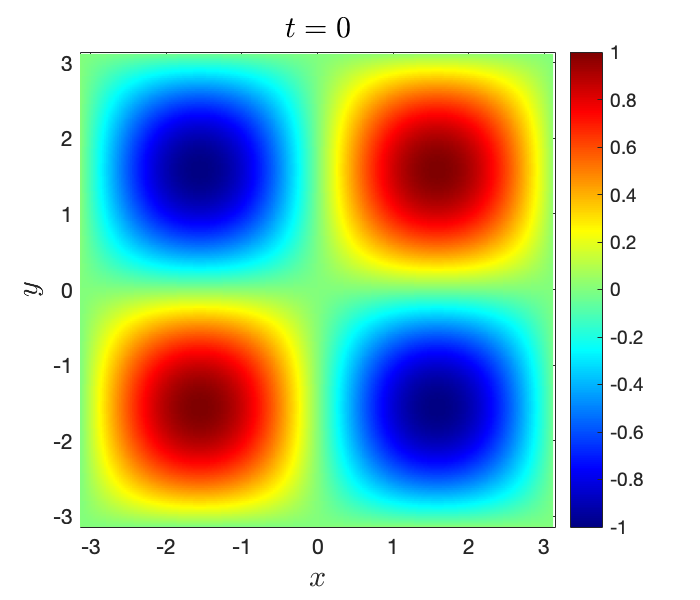}
\includegraphics[width=0.33\textwidth]{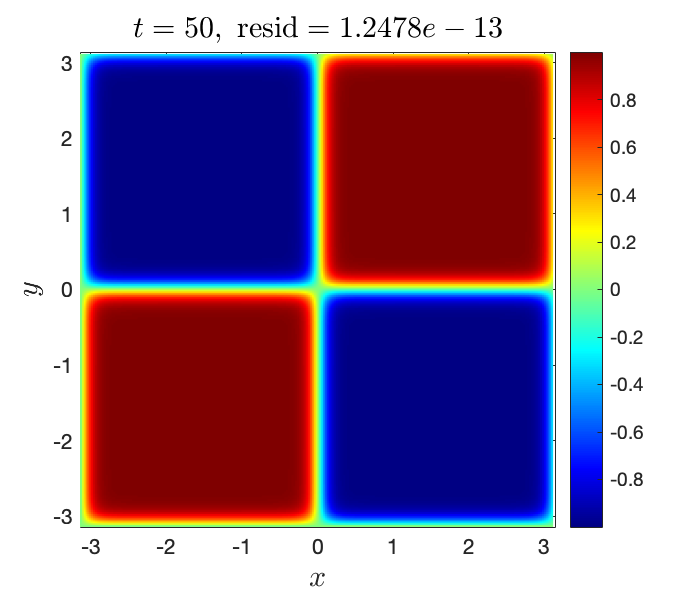}\\
\includegraphics[width=0.33\textwidth]{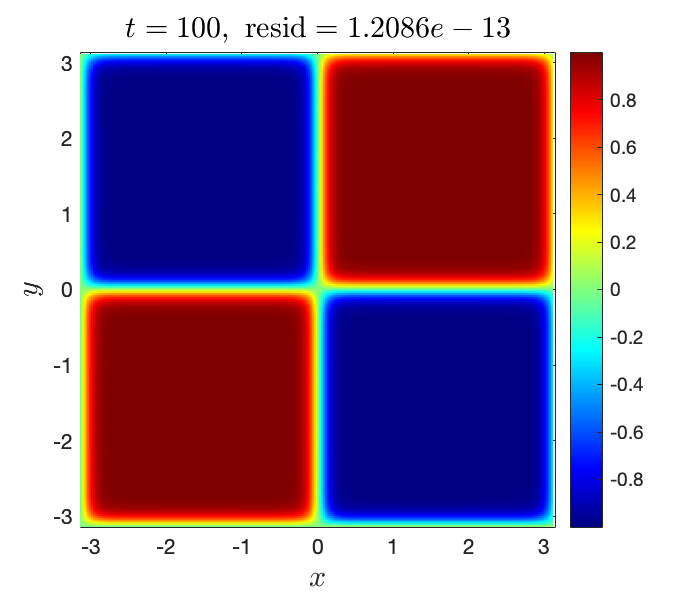}
\includegraphics[width=0.33\textwidth]{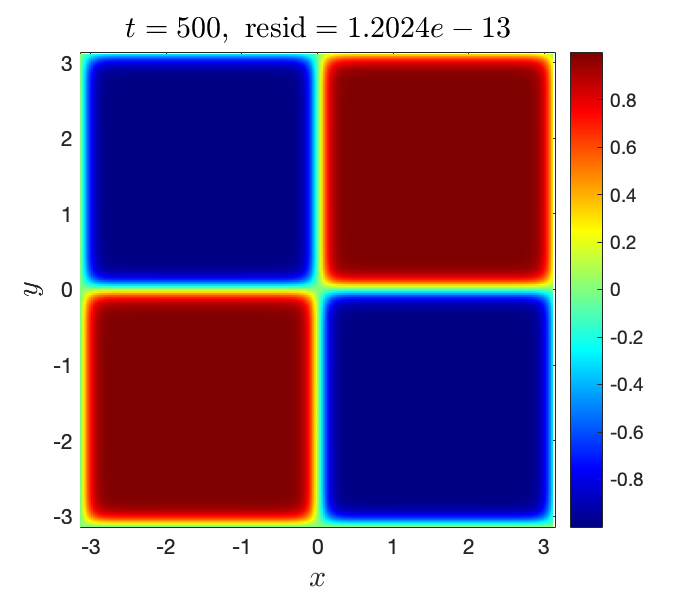}
\caption{\small Example \ref{exam4}: Same as Figure \ref{fig:ac2d1}, except with filter \eqref{5.3}. }\label{fig:ac2d2}
\end{figure}

Then, we impose the following 2D symmetry-preserving Fourier filter: $\forall\, n\geq 0,$
\begin{subequations} \label{5.3}
\begin{align}
&\widehat {u^n}(k_1=0,k_2) = 0,\quad \widehat {u^n}(k_1,k_2=0) = 0;\label{5.3-0}\\
& \operatorname{Im} (\widehat{u^n}(k_1,k_2) )=0;\label{5.4-0}\\
& \widehat {u^n}(k_1,k_2) \leftarrow \frac 1 2(\widehat {u^n}(k_1,k_2) -\widehat {u^n}(-k_1,k_2)), \quad \forall k_1,k_2;\label{5.5-0}\\
& \widehat {u^n}(k_1,k_2) \leftarrow \frac 1 2(\widehat {u^n}(k_1,k_2) -\widehat {u^n}(k_1,-k_2)), \quad \forall k_1,k_2.\label{5.6-0}
\end{align}
\end{subequations}
In the actual numerical implementation, we apply first \eqref{5.5-0} and then  \eqref{5.6-0}.
In {\sc Matlab} implementation, we use the following command

\vskip .25cm
\noindent
{\texttt{for j = 1:Ny\\
    un\_hat(:,j) = [0;(u\_hat(2:end,j)-flip(u\_hat(2:end,j)))/2];\\
end\\
for i = 1:Nx\\
    u\_hat(i,:) = [0,(un\_hat(i,2:end)-flip(un\_hat(i,2:end)))/2];\\
end}
}

\vskip .25cm
It is observed from Figure \ref{fig:ac2d1} that without filter the symmetries are destroyed and $u$ tends to the incorrect steady state.
On the other hand, the symmetry can be preserved using this filter, see Figure \ref{fig:ac2d2}.
A comparison between the corresponding energy evolutions is provided in Figure \ref{fig:2DAC_3}.

\begin{figure}[!h]
\centering
\includegraphics[width=0.48\textwidth]{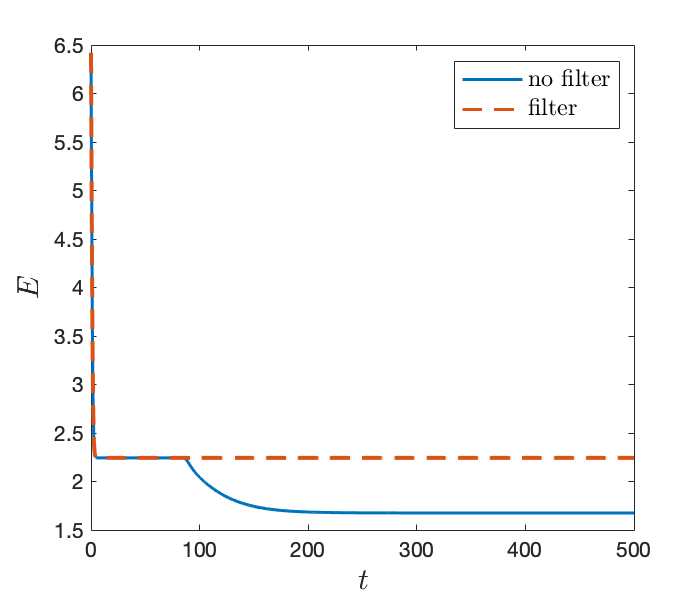}
\caption{\small Example \ref{exam4}: Energy evolutions of 2D Allen-Cahn equation computed by the first-order IMEX method  respectively without filter and with filter. }\label{fig:2DAC_3}
\end{figure}

\section{Concluding remarks}
In this work, we investigated the numerical computation and analysis for
steady states of the Allen-Cahn equation with
double well potential. Quite surprisingly we find that for very smooth odd initial
condition, the numerical solution computed using standard algorithms and very
high precision may lose parity in not very long time
simulations and eventually converge to the spurious steady state.  We call it
(numerical) breakdown of parity for Allen-Cahn as a manifestation of the gradual
accumulation of machine round-off errors. This new phenomenon shows that,
just much like other physical conservation laws, conservation of parity should be
ironed into the fundamental construction of numerical algorithms since the violation
of parity could lead to completely erroneous long time simulations.
To resolve this issue  we introduced a new parity-preserving filter technique
in order to give an accurate long time computation of the steady states. We developed a new theoretical framework taking
into account of perturbations introduced by machine errors. This opens doors to many future
directions and developments.

In future works, we shall develop
a new asymptotic theory and spectral analysis framework around the steady states.
We also plan to address the following important issues.

\begin{itemize}		
		\item General filtering/stabilizing technique for long time numerical computations. In our work we introduced a parity-preserving filter to deal with special data with odd symmetry or certain spectral band gaps. In the future we will develop a more systematic filtering procedure to suppress the spurious unstable modes in the long time simulations. In addition, we plan to develop
a complete theoretical framework to show the stability of filtered solutions under perturbations
by machine errors. This will include full numerical discretization as well as very realistic assumptions
on the machine errors and so on.

		\item  Other equations and phase field models. These include nonlocal Allen-Cahn equations driven by general polynomials or logarithmic nonlinearities, and also the Cahn-Hilliard equations, molecular beam epitaxy equation and so on.
	\end{itemize}

{\bf Acknowledgement.}
The research of T. Tang is partially supported by NSFC Grant 11731006, the NSFC/RGC Grant 11961160718.
The research of D. Li is supported in part by Hong Kong RGC grant GRF 16307317 and 16309518.
The research of W. Yang is supported by NSFC Grants 11801550 and 11871470.
The research of C. Quan is supported by NSFC Grant 11901281 and the Guangdong Basic and Applied Basic Research Foundation (2020A1515010336).


\bibliographystyle{plain}
\bibliography{AC_Num_v18_2.bib}

\end{document}